\documentclass[final]{amsart}
\usepackage[english]{babel}
\usepackage[utf8]{inputenc}
\usepackage{amsfonts}
\usepackage{amsmath}
\usepackage{amssymb}
\usepackage{amsthm}
\usepackage{graphicx}
\usepackage{mathrsfs}
\usepackage[all]{xy}
\usepackage[notcite,notref]{showkeys}
\usepackage{tikz}
\usepackage{footmisc}
\usepackage{multicol}
\usepackage{verbatim}
\usepackage{bbold}

	\theoremstyle{plain}
	\newtheorem{theorem}{Theorem}
	\newtheorem{lemma}{Lemma}
	\newtheorem{proposition}{Proposition}
	\newtheorem{corollary}{Corollary}
	
	\theoremstyle{definition}
	\newtheorem{definition}{Definition}
	\newtheorem{remark}{Remark}
	\newtheorem{example}{Example}

\newcommand{\Prim}{\mathrm{Prim}}

\makeindex

\title{Graph $C^\ast$-algebras with a $T_1$ primitive ideal space}

\author{James Gabe}

\address{Department of Mathematical Sciences, University of Copenhagen, Uni\-versi\-tets\-parken~5, DK-2100 Copenhagen, Denmark}
\email{gabe@math.ku.dk}
\thanks{Supported by the Danish National Research Foundation through the Centre for Symmetry and Deformation (DNRF92)}

\begin{document}
\maketitle

\begin{abstract}
We give necessary and sufficient conditions which a graph should satisfy in order for its associated $C^\ast$-algebra to have a $T_1$ primitive ideal space. We give a description of which one-point sets in such a primitive ideal space are open, and use this to prove that any \index{purely infinite}purely infinite graph $C^\ast$-algebra with a $T_1$ (in particular Hausdorff) primitive ideal space, is a $c_0$-direct sum of \index{Kirchberg algebra}Kirchberg algebras. Moreover, we show that graph $C^\ast$-algebras with a $T_1$ primitive ideal space canonically may be given the structure of a $C(\widetilde{\mathbb N})$-algebra, and that isomorphisms of their $\widetilde{\mathbb N}$-filtered $K$-theory (without coefficients) lift to $E(\widetilde{\mathbb N})$-equivalences, as defined by Dadarlat and Meyer.
\end{abstract}

\section{Introduction}\label{gabe:sec:intro}

When classifying non-simple $C^\ast$-algebras a lot of focus has been on $C^\ast$-algebras with finitely many ideals. However, Dadarlat and Meyer recently proved in \cite{gabe:DM} a Universal Multicoefficient Theorem in equivariant $E$-theory for separable $C^\ast$-algebras over second countable, zero-dimensional, compact Hausdorff spaces. In particular, together with the strong classification result of Kirchberg \cite{gabe:Kirchberg}, this shows that any separable, nuclear, $\mathscr O_\infty$-absorbing $C^\ast$-algebra with a zero-dimensional, compact Hausdorff primitive ideal space, for which all simple subquotients are in the classical bootstrap class, is strongly classified by its filtered total $K$-theory. This suggests and motivates the study of $C^\ast$-algebras with infinitely many ideals, in the eyes of classification.

In this paper we consider graph $C^\ast$-algebras with a $T_1$ primitive ideal space, i.e. a primitive ideal space in which every one-point set is closed. Clearly our main interest are such graph $C^\ast$-algebras with infinitely many ideals, since any finite $T_1$ space is discrete. In Section \ref{gabe:sec:prel} we recall the definition of graph $C^\ast$-algebras and many of the related basic concepts. In particular, we give a complete description of the primitive ideal space of a graph $C^\ast$-algebra. In Section \ref{gabe:sec:t1} we find necessary and sufficient condition which a graph should satisfy in order for the induced $C^\ast$-algebra to have a $T_1$ primitive ideal space. In Section \ref{gabe:sec:clopen} we prove that a lot of subsets of such primitive ideal spaces are both closed and open. In particular, we give a complete description of when one-point sets are open. We use this to show that any purely infinite graph $C^\ast$-algebra with a $T_1$ primitive ideal space is a $c_0$-direct sum of Kirchberg algebras. Moreover, we show that any graph $C^\ast$-algebra with a $T_1$ primitive ideal space may be given a canonical structure of a (not necessarily continuous) $C(\widetilde{\mathbb N})$-algebra, where $\widetilde{\mathbb N}$ is the one-point compactification of $\mathbb N$. As an ending remark, we prove that $\widetilde{\mathbb N}$-filtered $K$-theory classifies these $C(\widetilde{\mathbb N})$-algebras up to $E(\widetilde{\mathbb N})$-equivalence, as defined by Dadarlat and Meyer in \cite{gabe:DM}. 

\subsection*{Acknowledgement}

The author would like to thank his PhD-advisers Søren Eilers and Ryszard Nest for valuable discussions.

\section{Preliminaries}\label{gabe:sec:prel}

We recall the definition of a \index{graph $C^\ast$-algebra}graph $C^\ast$-algebra and many related definitions and properties. Let $E=(E^0,E^1,r,s)$ be a countable directed graph, i.e. a graph with countably many vertices $E^0$, countably many edges $E^1$ and a range and source map $r,s \colon E^1\to E^0$ respectively. A vertex $v\in E^0$ is called a \index{sink}sink if $s^{-1}(v)=\emptyset$ and an \index{infinite emitter}infinite emitter if $|s^{-1}(v)|=\infty$. A graph with no infinite emitters is called \index{row-finite}row-finite.

We define the graph $C^\ast$-algebra of $E$, $C^\ast(E)$, to be the universal $C^\ast$-algebra generated by a family of mutually orthogonal projections $\{ p_v : v\in E^0\}$ and partial isometries with mutually orthogonal ranges $\{ s_e : e \in E^1\}$, subject to the following Cuntz-Krieger relations
\begin{enumerate}
\item $s_e^\ast s_e = p_{r(e)}$ for $e \in E^1$,
\item $s_e s_e^\ast \leq p_{s(e)}$ for $e\in E^1$,
\item $p_v = \sum_{e \in s^{-1}(v)} s_e s_e^\ast$ for $v\in E^0$ such that $0 < |s^{-1}(v)| < \infty$.
\end{enumerate}

By universality there is a gauge action $\gamma \colon \mathbb T \to \textrm{Aut}(C^\ast(E))$ such that $\gamma_z(p_v)=p_v$ and $\gamma_z (s_e) = z s_e$ for $v\in E^0, e \in E^1$ and $z \in \mathbb T$. An ideal in $C^\ast(E)$ is said to be \index{gauge-invariant ideal}gauge-invariant if is invariant under $\gamma$. All ideals are assumed to be two-sided and closed.

If $\alpha_1,\dots , \alpha_n$ are edges such that $r(\alpha_i) = s(\alpha_{i+1})$ for $i = 1,\dots ,n-1$, then we say that $\alpha=(\alpha_1,\dots, \alpha_n)$ is a path, with source $s(\alpha) = s(\alpha_1)$ and range $r(\alpha)=r(\alpha_n)$. A loop is a path of positive length such that the source and range coincide, and this vertex is called the base of the loop. A loop $\alpha$ is said to have an exit, if there exist $e\in E^1$ and $i=1,\dots ,n$ such that $s(e)=s(\alpha_i)$ but $e\neq \alpha_i$. A loop $\alpha$ is called simple if $s(\alpha_i)\neq s(\alpha_j)$ for $i\neq j$. A graph $E$ is said to have \index{condition (K)}condition (K) if each vertex $v\in E^0$ is the base of no (simple) loop or is the base of at least two simple loops. It turns out that a graph $E$ has condition (K) if and only if every ideal in $C^\ast(E)$ is gauge-invariant if and only if $C^\ast(E)$ has real rank zero.

For $v,w\in E^0$ we write $v\geq w$ if there is a path $\alpha$ with $s(\alpha) = v$ and $r(\alpha) = w$. A subset $H$ of $E^0$ is called \index{hereditary set} hereditary if $v\geq w$ and $v\in H$ implies that $w\in H$. A subset $H$ of $E^0$ is called \index{saturated set}saturated if whenever $v\in E^0$ satisfies $0<|s^{-1}(v)|<\infty$ and $r(s^{-1}(v))\subseteq H$ then $v\in H$. If $X$ is a subset of $E^0$ then we let $\Sigma H(X)$ denote the smallest hereditary and saturated set containing $X$. If $H$ is hereditary and saturated we define
\begin{eqnarray*}
H^\mathrm{fin}_\infty &=& \{ v \in E^0 \backslash H : |s^{-1}(v)| = \infty \textrm{ and } 0 < |s^{-1}(v) \cap r^{-1}(E^0\backslash H)| < \infty \}, \\
H^\emptyset_\infty &=& \{ v \in E^0 \backslash H : |s^{-1}(v)| = \infty \textrm{ and } s^{-1}(v) \cap r^{-1}(E^0\backslash H) = \emptyset \}.
\end{eqnarray*}
By \cite[Theorem 3.6]{gabe:BHRS} there is a one-to-one correspondence between pairs $(H,B)$, where $H\subseteq E^0$ is hereditary and saturated and $B\subseteq H^\mathrm{fin}_\infty$, and the gauge-invariant ideals of $C^\ast(E)$. In fact, this is a lattice isomorphism when the different sets are given certain lattice structures. The ideal corresponding to $(H,B)$ is denoted $J_{H,B}$ and if $B=\emptyset$ we denote it by $J_H$.

A non-empty subset $M\subseteq E^0$ is called a \index{maximal tail}maximal tail if the following three conditions are satisfied.
\begin{enumerate}
\item[(1)] If $v\in E^0, w\in M$ and $v\geq w$ then $v\in M$.
\item[(2)] If $v\in M$ and $0< |s^{-1}(v)|<\infty$ then there exists $e\in E^1$ such that $s(e) = v$ and $r(e) \in M$.
\item[(3)] For every $v,w\in M$ there exists $y\in M$ such that $v\geq y$ and $w\geq y$.
\end{enumerate}
Note that $E^0\backslash M$ is hereditary by (1) and saturated by (2). Moreover, by (3) it follows that $(E^0 \backslash M)^\emptyset_\infty$ is either empty or consists of exactly one vertex. We let $\mathscr M(E)$ denote the set of all maximal tails in $E$, and let $\mathscr M_\gamma(E)$ denote the set of all maximal tails $M$ in $E$ such that each loop in $M$ has an exit in $M$. We let $\mathscr M_\tau(E) = \mathscr M(E) \backslash \mathscr M_\gamma(E)$.

If $X\subseteq E^0$ then define
\[
\Omega(X) = \{ w \in E^0\backslash X : w \ngeq v \textrm{ for all } v \in X\}.
\]
Note that if $M$ is a maximal tail, then $\Omega(M) = E^0\backslash M$. For a vertex $v\in E^0$, $E\backslash \Omega(v)$ is a maximal tail if and only if $v$ is a sink, an infinite emitter or if $v$ is the base of a loop.

We define the set of \index{breaking vertex}breaking vertices to be
\[
BV(E) = \{ v \in E^0 : |s^{-1}(v)|= \infty \textrm{ and } 0<|s^{-1}(v) \backslash r^{-1}(\Omega(v))| < \infty\}.
\]
Hence an infinite emitter $v$ is a breaking vertex if and only if $v\in \Omega(v)^{\mathrm{fin}}_\infty$.

In \cite{gabe:HS} they define for each $N\in \mathscr M_\tau(E)$ and $t\in \mathbb T$ a (primitive) ideal $R_{N,t}$, and prove that there is a bijection
\[
\mathscr M_\gamma (E) \sqcup BV(E) \sqcup (\mathscr M_\tau(E) \times \mathbb T) \to \Prim C^\ast(E)
\]
given by
\begin{eqnarray*}
\mathscr M_\gamma(E) \ni M &\mapsto& J_{\Omega(M),\Omega(M)^\mathrm{fin}_\infty} \\
BV(E) \ni v &\mapsto& J_{\Omega(v),\Omega(v)^\mathrm{fin}_\infty \backslash \{ v\}} \\
\mathscr M_\tau (E) \times \mathbb T \ni (N,t) & \mapsto& R_{N,t}.
\end{eqnarray*}

In \cite{gabe:HS}, Hong and Szymański give a complete description of the hull-kernel topology on $\Prim C^\ast(E)$ in terms of the maximal tails and breaking vertices. In order to describe this we use the following notation. Whenever $M\in \mathscr M_\tau(E)$ there is a unique (up to cyclic permutation) simple loop in $M$ which generates $M$, and we denote by $L^0_M$ the set of all vertices in this. If $Y\subseteq \mathscr M_\tau(E)$ we let
\begin{eqnarray*}
Y_\mathrm{min} &:=& \{ U \in Y : \textrm{for all } U' \in Y, U'\neq U \textrm{ there is no path from } L^0_U \textrm{ to } L^0_{U'} \}, \\
Y_\infty &:=& \{ U \in Y : \textrm{for all } V\in Y_\mathrm{min} \textrm{ there is no path from } L^0_U \textrm{ to } L^0_V \}.
\end{eqnarray*}

Due to a minor mistake in \cite{gabe:HS} the description of the topology is however not entirely correct. We will give a correct description below and explain what goes wrong in the original proof in Remark \ref{gabe:HSmistake}.

\begin{theorem}[Hong-Szymański]\label{gabe:HSthm}
Let $E$ be a countable directed graph. Let $X\subseteq \mathscr M_\gamma(E), W\subseteq BV(E), Y\subseteq \mathscr M_\tau(E)$, and let $D(U)\subseteq \mathbb T$ for each $U\in Y$. If $M\in \mathscr M_\gamma(E), v\in BV(E), N\in \mathscr M_\tau(E)$, and $z \in \mathbb T$, then the following hold.
\begin{enumerate}
\item\label{gabe:HSthm1} $M\in \overline X$ if and only if one of the following three condition holds.
\begin{enumerate}
\item[(i)] $M \in X$,
\item[(ii)] $M\subseteq \bigcup X$ and $\Omega(M)^\emptyset_\infty=\emptyset$,
\item[(iii)] $M\subseteq \bigcup X$ and $|s^{-1}(\Omega(M)^\emptyset_\infty)\cap r^{-1}(\bigcup X) | = \infty$.
\end{enumerate}
\item $v\in \overline X$ if and only if $v\in \bigcup X$ and $|s^{-1}(v) \cap r^{-1}(\bigcup X)| = \infty$.
\item\label{gabe:HSthm3} $(N,z)\in \overline X$ if and only if $N\subseteq \bigcup X$.
\item $M\in \overline W$ if and only if either
\begin{enumerate}
\item[(i)] $M \subseteq E^0\backslash \bigcap_{w\in W} \Omega(w)$ and $\Omega(M)^\emptyset_\infty =\emptyset$, or
\item[(ii)] $M \subseteq E^0\backslash \bigcap_{w\in W} \Omega(w)$ and $|s^{-1}(\Omega(M)^\emptyset_\infty) \cap r^{-1}(E^0\backslash \bigcap_{w\in W} \Omega(w))|=\infty$.
\end{enumerate}
\item $v\in \overline W$ if and only if either
\begin{enumerate}
\item[(i)] $v\in W$, or
\item[(ii)] $v \in E^0\backslash \bigcap_{w\in W} \Omega(w)$ and $|s^{-1}(v)\cap r^{-1}(E^0\backslash \bigcap_{w\in W} \Omega(w))|=\infty$.
\end{enumerate}
\item $(N,z) \in \overline W$ if and only if $N\subseteq E^0\backslash \bigcap_{w\in W} \Omega(w)$.
\item\label{gabe:HSthm7} $M$ is in the closure of $\{ (U,t) : U\in Y , t\in D(U) \}$ if and only if one of the following four conditions holds.
\begin{enumerate}
\item[(i)] $M\subseteq \bigcup Y_\infty$ and $\Omega(M)^\emptyset_\infty = \emptyset$,
\item[(ii)] $M\subseteq \bigcup Y_\infty$ and $|s^{-1}(\Omega(M)^\emptyset_\infty) \cap r^{-1}(\bigcup Y_\infty)|=\infty$,
\item[(iii)] $M\subseteq \bigcup Y_\mathrm{min}$ and $\Omega(M)^\emptyset_\infty = \emptyset$,
\item[(iv)] $M\subseteq \bigcup Y_\mathrm{min}$ and $|s^{-1}(\Omega(M)^\emptyset_\infty) \cap r^{-1}(\bigcup Y_\mathrm{min})|=\infty$.
\end{enumerate}
\item $v$ is in the closure of $\{ (U,t) : U\in Y , t\in D(U) \}$ if and only if either
\begin{enumerate}
\item[(i)] $v\in \bigcup Y_\infty$ and $|s^{-1}(v) \cap r^{-1}(\bigcup Y_\infty)| = \infty$, or
\item[(ii)] $v\in \bigcup Y_\mathrm{min}$ and $|s^{-1}(v) \cap r^{-1}(\bigcup Y_\mathrm{min})| = \infty$.
\end{enumerate}
\item\label{gabe:HSthm9} $(N,z)$ is in the closure of $\{ (U,t) : U\in Y , t\in D(U) \}$ if and only one of the following three conditions holds.
\begin{enumerate}
\item[(i)] $N \subseteq \bigcup Y_\infty$,
\item[(ii)] $N\notin Y_\mathrm{min}$ and $N\subseteq \bigcup Y_\mathrm{min}$,
\item[(iii)] $N\in Y_\mathrm{min}$ and $z \in \overline{D(N)}$.
\end{enumerate}
\end{enumerate}
\end{theorem}

\begin{remark}\label{gabe:HSmistake}
The minor mistake in the original proof of Theorem \ref{gabe:HSthm} is an error which occurs in the proofs of Lemma 3.3 and Theorem 3.4 of \cite{gabe:HS}. We will explain what goes wrong. Suppose that $M$ is a maximal tail, $K$ is a hereditary and saturated set such that $K \subseteq \Omega(M)$, and that $B\subseteq K^\mathrm{fin}_\infty$. Note that $B \backslash \Omega(M)^\emptyset_\infty \subseteq \Omega(M) \cup \Omega(M)^\mathrm{fin}_\infty$. Hence if $w\in \Omega(M)^\emptyset_\infty$ then $J_{K,B} \subseteq J_{\Omega(M),\Omega(M)^\mathrm{fin}_\infty}$ if and only $w \notin B$, since $w\notin \Omega(M)\cup \Omega(M)^\mathrm{fin}_\infty$. In the cases we consider we have that $w\in B$ if and only if $w\in K^\mathrm{fin}_\infty$. Now it is claimed that $w \notin K^\mathrm{fin}_\infty$ if and only $s^{-1}(w) \cap r^{-1}(K)$ is finite. However, this is not the case. If both $s^{-1}(w) \cap r^{-1}(K)$ and $s^{-1}(w) \cap r^{-1}(E^0\backslash K)$ are infinite then $w\notin K^\mathrm{fin}_\infty$. The correct statement would be that $w\notin K^\mathrm{fin}_\infty$ if and only if $|s^{-1}(w) \cap r^{-1}(E^0\backslash K)|= \infty$.

A similar thing occurs in the case where $v\in BV(E)$. Here we have, in the cases we consider, that $J_{K,B} \subseteq J_{\Omega(v),\Omega(v)^\mathrm{fin}_\infty\backslash \{ v \}}$ if and only if $v \notin K^\mathrm{fin}_\infty$. Again, the correct statement becomes $v\notin K^\mathrm{fin}_\infty$ if and only if $|s^{-1}(v) \cap r^{-1}(E^0 \backslash K) | = \infty$.

After changing these minor mistakes, one obtains Theorem \ref{gabe:HSthm} above.
\end{remark}

\section{$T_1$ primitive ideal space}\label{gabe:sec:t1}

Recall that a topological space is said to satisfy the separation axiom $T_1$ if every one-point set is closed. In particular, every Hausdorff space is a $T_1$ space. For a $C^\ast$-algebra $A$ the primitive ideal space $\Prim A$ is $T_1$ exactly if every primitive ideal is a maximal ideal. All of our ideals are assumed to be two-sided and closed.

As shown in \cite{gabe:BHRS}, every gauge-invariant primitive ideal of a graph $C^\ast$-algebra may be represented by a maximal tail or by a breaking vertex. The following lemma shows that we only need to consider maximal tails.

\begin{lemma}\label{gabe:nobreaking}
Let $E$ be a graph such that $\Prim(C^\ast(E))$ is $T_1$. Then $E$ has no breaking vertices.
\end{lemma}
\begin{proof}
Suppose $E$ has a breaking vertex $v$. Then 
\[
J_{\Omega(v),\Omega(v)_\infty^\mathrm{fin}\backslash\{v\}} \text{ and } J_{\Omega(v),\Omega(v)_\infty^\mathrm{fin}}
\]
are primitive ideals of $C^\ast(E)$, the former being a proper ideal of the latter by \cite[Corollary 3.10]{gabe:BHRS}. Hence
\[
J_{\Omega(v),\Omega(v)_\infty^\mathrm{fin}} \in \overline{\{J_{\Omega(v),\Omega(v)_\infty^\mathrm{fin}\backslash\{v\}}\}}
\]
and thus $C^\ast(E)$ can not have a $T_1$ primitive ideal space.
\end{proof}

It turns out that it might be helpful to consider gauge-invariant ideals which are maximal in the following sense.

\begin{definition}
Let $E$ be a countable directed graph and let $J$ be a proper ideal of $C^\ast(E)$. We say that $J$ is a \emph{maximal gauge-invariant ideal} if $J$ is gauge-invariant and if $J$ and $C^\ast(E)$ are the only gauge-invariant ideals containing $J$.
\end{definition}

The following theorem gives a complete description of the graphs whose induced $C^\ast$-algebras have a $T_1$ primitive ideal space.

\begin{theorem}\label{gabe:t1}
Let $E$ be a countable directed graph. The following are equivalent.
\begin{enumerate}
\item[(1)] $C^\ast(E)$ has a $T_1$ primitive ideal space,
\item[(2)] $E$ has no breaking vertices, and whenever $M$ and $N$ are maximal tails such that $M$ is a proper subset of $N$, then $\Omega(M)^\emptyset_\infty$ is non-empty, and
\[
|s^{-1}(\Omega(M)^\emptyset_\infty ) \cap r^{-1}(N) | < \infty,
\]
\item[(3)] $E$ has no breaking vertices, and $J_{\Omega(M),\Omega(M)^\mathrm{fin}_\infty}$ is a maximal gauge-invariant ideal in $C^\ast(E)$ for any maximal tail $M$,
\item[(4)] $E$ has no breaking vertices, and the map $M \mapsto J_{\Omega(M),\Omega(M)^\mathrm{fin}_\infty}$ is a bijective map from the set of maximal tails of $E$ onto the set of all maximal gauge-invariant ideals of $C^\ast(E)$.
\end{enumerate}
\end{theorem}

The last condition in (2) of the theorem may look complicated but it is easy to describe. It says, that if $M\subsetneq N$ are maximal tails then $M$ must contain an infinite emitter $v$ which only emits edges out of $M$, and only emits finitely many edges to $N$. Note that this is equivalent to $v\in \Omega(N)^\mathrm{fin}_\infty$. 

\begin{proof}
We start by proving $(1) \Leftrightarrow (2)$. By Lemma \ref{gabe:nobreaking} we may restrict to the case where $E$ has no breaking vertices. The proof is just a translation of Theorem \ref{gabe:HSthm} into our setting. We have four cases.

Case 1: Let $M,N \in \mathscr M_\gamma(E)$. By Theorem \ref{gabe:HSthm}.\ref{gabe:HSthm1} we have $M\in \overline{\{N\}}$ if and only if one of the following three holds: $(i)$ $M=N$, $(ii)$ $M\subsetneq N$ and $\Omega(M)^\emptyset_\infty = \emptyset$, $(iii)$ $M \subsetneq N$, $\Omega(M)^\emptyset_\infty\neq \emptyset$ and
\[
|s^{-1}(\Omega(M)_\infty^\emptyset) \cap r^{-1}(N)|= \infty.
\]
We eliminate the possibilities $(ii)$ and $(iii)$ exactly by imposing the conditions in (2).

Case 2: Let $(M,z) \in \mathscr M_\tau(E) \times \mathbb T$ and $N \in \mathscr M_\gamma(E)$. By Theorem \ref{gabe:HSthm}.\ref{gabe:HSthm3}, $(M,z)\in \overline{\{N\}}$ if and only if $M\subseteq N$. Since $M\in \mathscr M_\tau(E)$ it follows that $\Omega(M)^\emptyset_\infty = \emptyset$ and thus the conditions in (2) says $M \nsubseteq N$.

Case 3: Let $(N,t) \in \mathscr M_\tau(E) \times \mathbb T$ and $M\in \mathscr M_\gamma(E)$. Note that $\{ N\}_{\min} = \{ N\}$ and $\{N\}_\infty = \emptyset$. By Theorem \ref{gabe:HSthm}.\ref{gabe:HSthm7} we have $M\in \overline{\{(N,t)\}}$ if and only if one of the following two holds: $(i)$ $M\subseteq N$ and $\Omega(M)^\emptyset_\infty = \emptyset$, $(ii)$ $M\subseteq N$, $\Omega(M)^\emptyset_\infty\neq \emptyset$  and
\[
|s^{-1}(\Omega(M)_\infty^\emptyset) \cap r^{-1}(N)| = \infty.
\]
Conditions $(i)$ and $(ii)$ do not hold exactly when assuming the conditions of (2).

Case 4: Let $(M,z),(N,t)\in \mathscr M_\tau(E) \times \mathbb T$. By Theorem \ref{gabe:HSthm}.\ref{gabe:HSthm9} we have $(M,z)\in \overline{\{(N,t)\}}$ if and only if either $M\subsetneq N$, or $M=N$ and $z=t$. Note that condition $9(i)$ of the theorem can never be satisfied. Since the maximal tail $M$ satisfies $\Omega(M)^\emptyset_\infty = \emptyset$ the conditions of (2) say $M\subseteq N$ if and only if $M=N$ thus finishing $(1)\Leftrightarrow (2)$.

We will prove $(1) \Rightarrow (3)$. In order to simplify matters, we replace $E$ with its desingulisation $F$ (see \cite{gabe:DT}) thus obtaining a row-finite graph without sinks. Since $E$ has no breaking vertices by Lemma \ref{gabe:nobreaking}, there is a canonical one-to-one correspondence between $\mathscr M(E)$ and $\mathscr M(F)$ and a lattice isomorphism between the ideal lattices of $C^\ast(E)$ and $C^\ast(F)$ such that $M'\mapsto M$ implies $J_{\Omega(M'),\Omega(M')^\mathrm{fin}_\infty}\mapsto J_{\Omega(M)}$. In this case $J_{\Omega(M'),\Omega(M')^\mathrm{fin}_\infty}$ is a maximal gauge-invariant ideal if and only if $J_{\Omega(M)}$ is a maximal gauge-invariant ideal and thus it suffices to prove that $J_{\Omega(M)}$ is a maximal gauge-invariant ideal in $C^\ast(F)$ for $M\in \mathscr M(F)$. 

Suppose $J_{\Omega(M)} \subseteq J_H$ for some hereditary and saturated set $H\neq F^0$. Since $F$ is row-finite without sinks we may find an infinite path $\alpha$ in $F\backslash H$. Let
\[
N = \{ v \in F : v \geq s(\alpha_j) \textrm{ for some } j\}
\]
which is a maximal tail such that $N \subseteq F^0 \backslash H$. Hence $\Omega(M) \subseteq H \subseteq \Omega(N)$ which implies $N\subseteq M$. Since $F$ is row-finite, $\Omega(N)^\emptyset_\infty$ is empty, and thus since $(1) \Leftrightarrow (2)$, $M=N$. Hence $H=\Omega(M)$ and thus $(1) \Rightarrow (3)$.

We will prove $(3)\Rightarrow (4)$. Again, we let $F$ be the desingulisation of $E$ and note that $(4)$ holds for $F$ if and only if it holds for $E$. Note that $(3)$ implies that the map in $(4)$ is well-defined, and this is clearly injective. Let $H$ be a hereditary and saturated set in $F$ such that $J_H$ is a maximal gauge-invariant ideal in $C^\ast(F)$. As above, we may find a maximal tail $M$ such that $H\subseteq \Omega(M)$ which implies $J_H \subseteq J_{\Omega(M)}$. Since $J_H$ is a maximal gauge-invariant ideal, $H= \Omega(M)$ which proves surjectivity of the map and finishes $(3) \Rightarrow (4)$.

For $(4) \Rightarrow (1)$ we may again replace $E$ by its desingulisation $F$. Since $(1)\Leftrightarrow (2)$ and $F$ is row-finite, $(1)$ is equivalent to the following: if $M\subseteq N$ are maximal tails then $M=N$, since $\Omega(M)^\emptyset_\infty = \empty$ for every maximal tail $M$. Let $M\subseteq N$ be maximal tails in $F$. Then $J_{\Omega(N)}\subseteq J_{\Omega(M)}$ are maximal gauge-invariant ideals and thus $N=M$, which finishes the proof.
\end{proof}

\begin{definition}
Let $E$ be a countable directed graph. If $E$ satisfies one (and hence all) of the conditions in Theorem \ref{gabe:t1}, then we say that $E$ is a \emph{$T_1$ graph}.
\end{definition}

For row-finite graphs the above theorem simplifies significantly.

\begin{corollary}\label{gabe:t1rf}
Let $E$ be a row-finite graph. The following are equivalent.
\begin{enumerate}
\item[(1)] $E$ is a $T_1$ graph,
\item[(2)] if $M\subseteq N$ are maximal tails, then $M=N$,
\item[(3)] $J_{\Omega(M)}$ is a maximal gauge-invariant ideal in $C^\ast(E)$ for any maximal tail $M$,
\item[(4)] the map $M \mapsto J_{\Omega(M)}$ is a bijective map from the set of maximal tails of $E$ onto the set of all maximal gauge-invariant ideals of $C^\ast(E)$,
\end{enumerate}
\end{corollary}

\begin{proof}
Since $E$ is row-finite it has no breaking vertices and $\Omega(M)^\emptyset_\infty$ is empty for any maximal tail $M$. Hence it follows from Theorem \ref{gabe:t1}.
\end{proof}

We will end this section by constructing a class of graph $C^\ast$-algebras, all of which have a non-discrete $T_1$ primitive ideal space.

\begin{example}\label{gabe:exmclass}
Let $B$ be a simple AF-algebra and let $F$ be a Bratteli diagram of $B$ as in \cite{gabe:Drinen}, such that the vertex set $F^0$ is partitioned into vertex sets $F^0_n = \{ w_n^1 , \dots ,w^{k_n}_n\}$ and every edge with a source in $F^0_n$ has range in $F^0_{n+1}$. Let $G_1,G_2,\dots$ be a sequence of graphs all of which have no non-trivial hereditary and saturated sets. Construct a graph $E$ as follows:
\begin{eqnarray*}
E^0 &=& F^0 \cup \bigcup_{n=1}^\infty G^0_n, \\
E^1 &=& F^1 \cup \bigcup_{n=1}^\infty G^1_n \cup \bigcup_{n=1}^\infty \{ e^1_n, \dots e^{k_n}_n\}
\end{eqnarray*}
where the range and source maps do not change on $F^1 \cup \bigcup_{n=1}^\infty G^1_n$ and where $s(e^j_n) = w^j_n$ and $r(e^j_n)\in G^0_n$.

Using that $F$ and each $G_n$ have no non-trivial hereditary and saturated sets we get that the maximal tails of $E$ are
\begin{eqnarray*}
M_n &=& \bigcup_{k=1}^n F^0_k \cup G^0_n, \\
M_\infty &=& \bigcup_{k=1}^\infty F^0_k = F^0.
\end{eqnarray*}
Hence no maximal tail is contained in another and thus the primitive ideal space of $C^\ast(E)$ is $T_1$. For any of these maximal tails $M$, each vertex in $M$ emits only finitely many edges to $\Omega(M)$ and thus $\Omega(M)^\mathrm{fin}_\infty$ is empty. The quotients $C^\ast(E)/J_{\Omega(M_n)}$ are Morita equivalent $C^\ast(G_n)$ and $C^\ast(E)/J_{\Omega(M_\infty)}=C^\ast(F)$ which is Morita equivalent to $B$.

If, in addition, each $G_n$ has condition (K) then one can verify that $\Prim C^\ast(E)$ is homeomorphic to $\widetilde{\mathbb N}=\mathbb N \cup \{\infty\}$, the one-point compactification of $\mathbb{N}$. Such a homeomorphism may be given by 
\[
\widetilde{\mathbb N} \ni n \mapsto J_{\Omega(M_n)} \in \Prim C^\ast(E).
\]
\end{example}

\section{Clopen maximal gauge-invariant ideals}\label{gabe:sec:clopen}

Whenever a subset of a topological space is both closed and open, then we say that the set is \emph{clopen}. In this section we give a description of which one-point sets in the primitive ideal space of a $T_1$ graph are clopen. In fact, we describe which maximal gauge-invariant ideals in the primitive ideal space correspond to clopen sets. We use this description to show that every purely infinite graph $C^\ast$-algebra with a $T_1$ primitive ideal space is a $c_0$-direct sum of Kirchberg algebras. Moreover, we prove that graph $C^\ast$-algebras with a $T_1$ primitive ideal space are canonically $C(\widetilde{\mathbb N})$-algebras, which are classified up to $E(\widetilde{\mathbb N})$-equivalence by their $\widetilde{\mathbb N}$-filtered $K$-theory.

In order to describe the clopen maximal gauge-invariant ideals, we need a notion of when a maximal tail distinguishes itself from all other maximal ideals in a certain way.
\begin{definition}
Let $E$ be a $T_1$ graph and let $M$ be a maximal tail in $E$. We say that $M$ is \emph{isolated} if either
\begin{enumerate}
\item $M$ contains a vertex which is not contained in any other maximal tail, or
\item $\Omega(M)^\emptyset_\infty$ is non-empty and 
\[
|s^{-1}(\Omega(M)^\emptyset_\infty) \cap r^{-1}\left( \bigcup_{N\in \mathscr M_{M}(E)} N \right) | < \infty.
\]
where $\mathscr M_{M} (E)$ denotes the set of all maximal tails $N$ such that $M\subseteq N$.
\end{enumerate}
\end{definition}

This definition may look strange but it turns out that a maximal tail corresponds to a clopen maximal gauge-invariant ideal if and only if it is isolated, see Theorem \ref{gabe:clopenpoints}.

\begin{remark}\label{gabe:isorf}
For a row-finite $T_1$ graph $E$ the above definition simplifies, since $\Omega(M)^\emptyset_\infty$ is empty for any maximal tail $M$. Hence, in this case, a maximal tail is isolated if and only if it contains a vertex which is not contained in any other maximal tail. 
\end{remark}

\begin{example}
Consider the two graphs
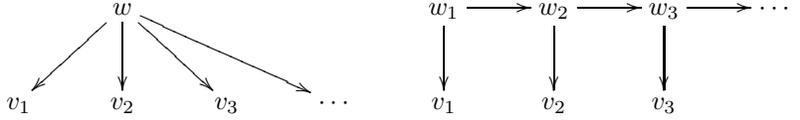
\[
\xymatrix{
& w \ar[dl] \ar[d] \ar[dr] \ar[drr] & & & w_1 \ar[r] \ar[d] & w_2 \ar[r] \ar[d] & w_3 \ar[r] \ar[d] & \cdots. \\
v_1 & v_2 & v_3 & \cdots & v_1 & v_2 & v_3 &
}
\]
The latter graph is the desingulisation of the former but without changing sinks to tails. The maximal tails of the former graph are given by $N_n = \{ w, v_n\}$ and $N_\infty = \{ w\}$. The maximal tails of the latter graph are
\begin{eqnarray*}
M_n &=& \{ w_1,\dots, w_n, v_n\}, \\
M_\infty &=& \{ w_1, w_2, \dots\}.
\end{eqnarray*}
Hence both graphs are easily seen to be $T_1$ graphs. All the maximal tails $N_n$ and $M_n$ for $n\in \mathbb N$ are easily seen to be isolated, and by Remark \ref{gabe:isorf}, $M_\infty$ is not isolated. Since $\Omega(N_\infty)=\{ w \}$ and
\[
|s^{-1}(v) \cap r^{-1} \left( \bigcup_{N\in \mathscr M_{N_\infty}(E)}N \right) | = \infty
\]
we note that $N_\infty$ is not isolated. In fact, by Corollary \ref{gabe:isodesing} below, $N_\infty$ would be isolated if and only if $M_\infty$ was isolated.

The latter graph is an example of a graph in Example \ref{gabe:exmclass}, with $B=\mathbb C$ and each $G_n$ consisting of one vertex and no edges. Since the graph has condition (K), the primitive ideal space is homeomorphic to $\widetilde{\mathbb N} = \mathbb N \cup \{ \infty \}$, the one-point compactification of $\mathbb N$, by the map
\[
\widetilde{\mathbb N} \ni n \mapsto J_{\Omega(M_n)}.
\]
\end{example}

It turns out that many maximal tails are isolated, as can be seen in the following lemma.

\begin{lemma}\label{gabe:maxtailloop}
Let $E$ be a $T_1$ graph and let $M$ be a maximal tail which contains a sink or a loop. Then $M$ is isolated.
\end{lemma}
\begin{proof}
Let $v\in M$ be the sink or the base of a loop in $M$, and note that $\Omega(v)^\emptyset_\infty = \emptyset$. If $N$ is a maximal tail such that $v\in N$ then $E^0 \backslash \Omega(v) \subseteq N$ and since $\Omega(v)^\emptyset_\infty$ is empty, $N=E^0\backslash \Omega(v)$ by Theorem \ref{gabe:t1}. Hence $v$ is not contained in any other maximal tail than $M$ and thus $M$ is isolated.
\end{proof}

The following is the main theorem of this section, mainly due to all the corollaries following it.

\begin{theorem}\label{gabe:clopenpoints}
Let $E$ be a countable directed graph for which the primitive ideal space of $C^\ast(E)$ is $T_1$, and let $M$ be a maximal tail in $E$. Then 
\[
\{ \mathfrak p \in \Prim C^\ast(E) : J_{\Omega(M),\Omega(M)^\mathrm{fin}_\infty} \subseteq \mathfrak p \} \subseteq \Prim C^\ast(E)
\]
is a clopen set if and only if $M$ is isolated.

In particular, if $M\in \mathscr M_\gamma (E)$, then the one-point set
\[
\{ J_{\Omega(M),\Omega(M)^\mathrm{fin}_\infty} \} \subseteq \Prim C^\ast(E)
\]
is clopen if and only if $M$ is isolated, and if $M\in \mathscr M_\tau(E)$ then
\[
\{ R_{M,t} : t \in \mathbb T\} \subseteq \Prim C^\ast(E)
\]
is a clopen set homeomorphic to the circle $S^1$.
\end{theorem}
\begin{proof}
To ease notation define
\begin{eqnarray*}
U_M &:=& \{ \mathfrak p \in \Prim C^\ast(E) : J_{\Omega(M)} \subseteq \mathfrak p \}.
\end{eqnarray*}
By definition $U_M$ is closed. By \cite[Lemma 2.6]{gabe:HS} it follows that if $J$ is a gauge-invariant ideal, $M\in \mathscr M_\tau(E)$ and $t\in \mathbb T$, then $J\subseteq R_{M,t}$ if and only if $J\subseteq J_{\Omega(M),\Omega(M)^\mathrm{fin}_\infty}$. We will use this fact several times throughout the proof, without mentioning it.

Suppose $U_M$ is clopen. If $M \in \mathscr M_\tau(E)$ then $M$ contains a loop and is thus isolated by Lemma \ref{gabe:maxtailloop}. Hence we may suppose $M\in \mathscr M_\gamma(E)$ for which $U_M = \{ J_{\Omega(M),\Omega(M)^\mathrm{fin}_\infty}\}$. Since $U_M$ is open there is a unique ideal $J$ such that
\[
\{ J_{\Omega(M),\Omega(M)^\mathrm{fin}_\infty} \} = \{ \mathfrak p \in \Prim C^\ast(E): J \nsubseteq \mathfrak p\}.
\]
Suppose $J$ is not gauge-invariant. Then we can find a $z \in \mathbb T$ such that $\gamma_z (J) \neq J$. Note that $\gamma_z(J) \nsubseteq \gamma_z(J_{\Omega(M),\Omega(M)^\mathrm{fin}_\infty}) = J_{\Omega(M),\Omega(M)^\mathrm{fin}_\infty}$. For any primitive ideal $\mathfrak p \in \Prim C^\ast(E) \backslash \{ J_{\Omega(M),\Omega(M)^\mathrm{fin}_\infty} \}$, we have $\gamma_z(J) \subseteq \gamma_z(\mathfrak p)$, since $J\subseteq \mathfrak p$. Since $\gamma_z$ fixes $J_{\Omega(M),\Omega(M)^\mathrm{fin}_\infty}$ it induces a bijection from $\Prim C^\ast(E) \backslash \{ J_{\Omega(M),\Omega(M)^\mathrm{fin}_\infty}\}$ to itself and thus $\gamma_z(J) \subseteq \mathfrak p$ for any primitive ideal $\mathfrak p \neq J_{\Omega(M),\Omega(M)^\mathrm{fin}_\infty}$. However, this contradicts the uniqueness of $J$, and thus $J$ must be gauge-invariant.

Since $J$ is gauge-invariant, $J=J_{H,B}$ for a hereditary and saturated set $H$ and $B\subseteq H^\mathrm{fin}_\infty$. If $H\nsubseteq \Omega(M)$ then any vertex $v\in H$ such that $v\in M$ is not contained in any other maximal tail, since $J_{H,B} \subseteq J_{\Omega(N),\Omega(N)^\mathrm{fin}_\infty}$ for any maximal tail $N\neq M$. Hence we may restrict to the case where $H \subseteq \Omega(M)$. Since $J_{H,B} \nsubseteq J_{\Omega(M),\Omega(M)^\mathrm{fin}_\infty}$, $B\nsubseteq \Omega(M) \cup \Omega(M)^\mathrm{fin}_\infty$. It is easily observed that $B \backslash \Omega(M)^\emptyset_\infty \subseteq \Omega(M)\cup \Omega(M)^\mathrm{fin}_\infty$ and hence it follows that $\Omega(M)^\emptyset_\infty = \{ w \}$ for some vertex $w$ and that $w\in B$. Recall that $\mathscr M_M(E) = \{ N\in \mathscr M(E):M\subseteq N\}$. Since $w\in H^\mathrm{fin}_\infty$ we have
\[
| s^{-1}(w) \cap r^{-1} (E^0 \backslash H ) | < \infty
\]
and since $\bigcup_{N\in \mathscr M_M(E)}N \subseteq E^0 \backslash H$ it follows that
\[
| s^{-1}(w) \cap r^{-1} (\bigcup_{N\in \mathscr M_M(E)} N ) | < \infty.
\]
Thus $M$ is isolated.

Now suppose that $M$ is an isolated maximal tail. If $M$ contains a vertex $v$ which is not contained in any other maximal tail, then $J_{\Sigma H(v)} \nsubseteq J_{\Omega(M),\Omega(M)^\mathrm{fin}_\infty}$ and $J_{\Sigma H(v)} \subseteq J_{\Omega(N),\Omega(N)^\mathrm{fin}_\infty}$ for any maximal tail $N\neq M$. Hence
\[
U_M = \{ \mathfrak p \in \Prim C^\ast(E) : J_{\Sigma H(v)} \nsubseteq \mathfrak p\}
\]
and thus $U_M$ is clopen. Now suppose that every vertex of $M$ is contained in some other maximal tail. Let $H = \bigcap_{N\in \mathscr M_M(E)} \Omega(N)$ which is hereditary and saturated. Since $M$ is isolated, $\Omega(M)^\emptyset_\infty = \{ w \}$ for some vertex $w$ and moreover $w \in H^\mathrm{fin}_\infty$. Hence $J_{H,\{ w \}} \nsubseteq J_{\Omega(M),\Omega(M)^\mathrm{fin}_\infty}$ and $J_{H,\{ w \}} \subseteq J_{\Omega(N),\Omega(N)^\mathrm{fin}_\infty}$ for any $N\in \mathscr M_M(E)$ by Theorem \ref{gabe:t1}. Now, as above, $J_{\Sigma H(w)} \subseteq J_{\Omega(N),\Omega(N)^\mathrm{fin}_\infty}$ for any $N\notin \mathscr M_M(E)$ and $J_{\Sigma H(w)} \nsubseteq J_{\Omega(N),\Omega(N)^\mathrm{fin}_\infty}$ for $N\in \mathscr M_M(E)$. Hence
\[
U_M = \{ \mathfrak p : J_{H,\{ w \}} \nsubseteq \mathfrak p \} \cap \{ \mathfrak p : J_{\Sigma H(w)} \nsubseteq \mathfrak p\}
\]
is the intersection of two open sets, and is thus clopen.

For the 'in particular' part note that if $M\in \mathscr M_\gamma(E)$ then $U_M=\{ J_{\Omega(M),\Omega(M)^\mathrm{fin}_\infty }\}$. If $M\in \mathscr M_\tau(E)$ then $M$ contains a loop and is thus isolated by Lemma \ref{gabe:maxtailloop}. Hence
\[
U_M = \{ R_{M,t} : t \in \mathbb T\}
\]
is clopen. By Theorem \ref{gabe:HSthm} it follows that this set is homeomorphic to the circle $S^1$.
\end{proof}

\begin{corollary}\label{gabe:onepointsets}
Let $E$ be a $T_1$ graph and $\mathfrak p \in \Prim C^\ast(E)$ be a primitive ideal. Then $\{ \mathfrak p \}$ is clopen if and only if $\mathfrak p = J_{\Omega(M),\Omega(M)^\mathrm{fin}_\infty}$ for an isolated maximal tail $M \in \mathscr M_\gamma (E)$.
\end{corollary}

\begin{corollary}\label{gabe:discrete}
Let $E$ be a $T_1$ graph and suppose that every maximal tail in $E$ is isolated. Then 
\[
\Prim C^\ast(E) \cong \bigsqcup_{M \in \mathscr M_\gamma} \star \sqcup \bigsqcup_{M\in \mathscr M_\tau} S^1
\]
is a disjoint union, where $\star$ is a one-point topological space and $S^1$ is the circle.

In particular, if $E$ in addition has condition (K) then $\Prim C^\ast(E)$ is discrete.
\end{corollary}

If two graphs $E$ and $F$ have Morita equivalent $C^\ast$-algebras, then the corresponding ideal lattices are canonically isomorphic. Hence, if $E$ and $F$ have no breaking vertices, there is an induced one-to-one correspondence between the maximal tails in $E$ and $F$. The following corollary is immediate from Theorem \ref{gabe:clopenpoints}.

\begin{corollary}\label{gabe:isodesing}
Let $E$ and $F$ be $T_1$ graphs such that $C^\ast(E)$ and $C^\ast(F)$ are Morita equivalent. Then a maximal tail in $E$ is isolated if and only if the corresponding maximal tail in $F$ is isolated.
\end{corollary}

Our main application of the above theorem is the following corollary.

\begin{corollary}\label{gabe:directsum}
Any purely infinite graph $C^\ast$-algebra with a $T_1$ (in particular Hausdorff) primitive ideal space is isomorphic to a $c_0$-direct sum of Kirchberg algebras.
\end{corollary}
\begin{proof}
Let $E$ be a $T_1$ graph such that $C^\ast(E)$ is purely infinite. By \cite[Theorem 2.3]{gabe:HSpi} $E$ has condition (K) and every maximal tail in $E$ contains a loop, and is thus isolated by Lemma \ref{gabe:maxtailloop}. By Corollary \ref{gabe:discrete} the primitive ideal space $\Prim C^\ast(E)$ is discrete. Hence $C^\ast(E)$ is the $c_0$-direct sum of all its simple ideals, which are Kirchberg algebras.
\end{proof}

We also have another application of the above theorem.

\begin{corollary}\label{gabe:quotientaf}
Let $A$ be a graph $C^\ast$-algebra for which the primitive ideal space is $T_1$. Let $J$ be the ideal generated by all the direct summands in $A$ corresponding to $A/J_{\Omega(M),\Omega(M)^\mathrm{fin}_\infty}$ where $M$ is an isolated maximal tail. Then $A/J$ is an AF-algebra.
\end{corollary}
\begin{proof}
Note that the ideal is well-defined by Theorem \ref{gabe:clopenpoints}, since $J_{\Omega(M),\Omega(M)^\mathrm{fin}_\infty}$ is a direct summand in $A$ for every isolated maximal tail $M$. By Corollary \ref{gabe:isodesing} it suffices to prove this up to Morita equivalence. Hence we may assume that there is a row-finite graph $E$ such that $C^\ast(E)=A$. Let $V$ denote the set of all vertices which are contained in exactly one maximal tail. For any isolated maximal tail $M$, the direct summand in $A$ which corresponds to $A/J_{\Omega(M)}$ is $J_{\Sigma H(v)}$ where $v$ is any vertex in $M$ which is not contained in any other maximal tail. Hence $J=J_{\Sigma H(V)}$ since this is the smallest ideal containing all $J_{\Sigma H(v)}$ for $v\in V$. By Lemma \ref{gabe:maxtailloop} any vertex which is the base of a loop, is in $V$. Hence the graph $E\backslash \Sigma H(V)$ contains no loops and thus $A/J = C^\ast(E\backslash \Sigma H(V))$ is an AF-algebra.
\end{proof}

\begin{remark}
By an analogous argument as given in the proof of Corollary \ref{gabe:quotientaf}, we get the following result. Let $A$ be a real rank zero graph $C^\ast$-algebra for which the primitive ideal space is $T_1$. Then $A$ contains a (unique) purely infinite ideal $J$ such that $A/J$ is an AF-algebra. 

In fact, we could define $V$ in the proof of Corollary \ref{gabe:quotientaf} to be the set of all vertices which are the base of some loop. Then $J=J_{\Sigma H(V)}$ would be the direct sum of all simple purely infinite ideals in $A$, and $A/J$ would again be an AF-algebra.
\end{remark}

\begin{remark}\label{gabe:C(N)}
Let $\widetilde {\mathbb N} = \mathbb N \cup \{ \infty \}$ be the one-point compactification of $\mathbb N$. We may give any graph $C^\ast$-algebra $A$ with a $T_1$ primitive ideal space a canonical structure of a $C(\widetilde{\mathbb N})$-algebra. In fact, list all of the direct summands in $A$ corresponding to $A/J_{\Omega(M),\Omega(M)^\mathrm{fin}_\infty}$ for $M$ an isolated maximal tail, as $J_1,J_2,\dots$. By letting
\[
A(\{ n\} ) = J_n, \text{ and } A(\{n,n+1,\dots,\infty\}) = A/\bigoplus_{k=1}^{n-1} J_k,
\]
then $A$ gets the structure of a $C^\ast$-algebra over $\widetilde{\mathbb N}$ which is the same as a (not necessarily continuous) $C(\widetilde{\mathbb N})$-algebra (see e.g. \cite{gabe:MN}). This structure is unique up to an automorphism functor $\sigma_*$ on $\mathfrak {C^\ast alg}(\widetilde{\mathbb N})$, the category of $C(\widetilde{\mathbb N})$-algebras, where $\sigma:\widetilde{\mathbb N}\to \widetilde{\mathbb N}$ is a homeomorphism. Moreover, by Corollary \ref{gabe:quotientaf}, the fibre $A_\infty$ is an AF-algebra.
\end{remark}

Using the structure of a $C(\widetilde{\mathbb N})$-algebra we may construct an $\widetilde{\mathbb N}$-filtered $K$-theory functor as in \cite{gabe:DM}. In fact, let $C(\widetilde{\mathbb N},\mathbb Z)$ be the ring of locally constant maps $\widetilde{\mathbb N} \to \mathbb Z$. If $A$ is a $C(\widetilde{\mathbb N})$-algebra then the $K$-theory $K_\ast(A)$ has the natural structure as a $\mathbb Z/2$-graded $C(\widetilde{\mathbb N},\mathbb Z)$-module. Similarly, let $\Lambda$ be the ring of Böckstein operation, and let $C(\widetilde{\mathbb N},\Lambda)$ be the ring of locally constant maps $\widetilde{\mathbb N} \to \Lambda$. If $A$ is a $C(\widetilde{\mathbb N})$-algebra then the total $K$-theory $\underline K(A)$ has the natural structure as a $C(\widetilde{\mathbb N},\Lambda)$-module. It is this latter invariant, that Dadarlat and Meyer proved a UMCT for.

We end this paper by showing that for $T_1$ graph $C^\ast$-algebras given $C(\widetilde{\mathbb N})$-algebra structures as in Remark \ref{gabe:C(N)}, an isomorphism of $\widetilde{\mathbb N}$-filtered $K$-theory (without coefficients) lifts to an $E(\widetilde{\mathbb N})$-equivalence. Note that this is not true in general by \cite[Example 6.14]{gabe:DM}.

\begin{proposition}\label{gabe:e-quivalence}
Let $A$ and $B$ be graph $C^\ast$-algebras with $T_1$ primitive ideal spaces, and suppose that these have the structure of $C(\widetilde{\mathbb N})$-algebras as in Remark \ref{gabe:C(N)}. Then $K_\ast(A) \cong K_\ast(B)$ as $\mathbb Z/2$-graded $C(\widetilde{\mathbb N},\mathbb Z)$-modules if and only if $A$ and $B$ are $E(\widetilde{\mathbb N})$-equivalent.

In addition, if $A$ and $B$ are continuous $C(\widetilde{\mathbb N})$-algebras, then $K_\ast(A) \cong K_\ast(B)$ as $\mathbb Z/2$-graded $C(\widetilde{\mathbb N},\mathbb Z)$-modules if and only if $A$ and $B$ are $KK(\widetilde{\mathbb N})$-equivalent.
\end{proposition}
\begin{proof}
Clearly an $E(\widetilde{\mathbb N})$-equivalence induces an isomorphism of $\widetilde{\mathbb N}$-filtered $K$-theory. Suppose that $\phi=(\phi_0,\phi_1) \colon K_\ast(A) \to K_\ast(B)$ is an isomorphism of $\mathbb Z/2$-graded $C(\widetilde{\mathbb N},\mathbb Z)$-modules. By the UMCT of Dadarlat and Meyer, \cite[Theorem 6.11]{gabe:DM}, it suffices to lift $\phi$ to an isomorphism of $\widetilde{\mathbb N}$-filtered total $K$-theory. Since the $K_1$-groups are free, $K_0(D;\mathbb Z/n) = K_0(D)\otimes \mathbb Z/n$ for $D\in \{ A,B\}$. Hence define 
\[
\phi^n_0 = \phi_0 \otimes id_{\mathbb Z/n} \colon K_0(A;\mathbb Z/n) \to K_0(B;\mathbb Z/n)
\]
which are isomorphisms for each $n\in \mathbb N$. Since the fibres $A_\infty$ and $B_\infty$ are $AF$-algebras by Corollary \ref{gabe:quotientaf}, $K_1(A_\infty;\mathbb Z/n) = K_1(B_\infty;\mathbb Z/n)=0$ for each $n\in \mathbb N$. Since the map $K_0(D;\mathbb Z/n) \to K_0(D_\infty;\mathbb Z/n)$ is clearly surjective, and $K_1(D_\infty;\mathbb Z/n)=0$, it follows by six-term exactness that 
\[
K_1(D;\mathbb Z/n) \cong K_1(D(\mathbb N);\mathbb Z/n) \cong \bigoplus_{k\in \mathbb N} K_1(D_k;\mathbb Z/n)
\]
for $D\in \{ A,B\}$ and $n\in \mathbb N$. Since $\phi_\ast\colon K_\ast(A)\to K_\ast(B)$ is an isomorphism of $\mathbb Z/2$-graded $C(\widetilde{\mathbb N},\mathbb Z)$-modules, $\phi_\ast$ restricts to an isomorphism $\phi_{\ast,k} \colon K_\ast(A_k) \to K_\ast(B_k)$ for each $k\in \mathbb N$. Lift these to isomorphisms of the total $K$-theory $\underline{\phi_{\ast,k}} \colon \underline K(A_k) \to \underline K(B_k)$. Now define the group isomorphisms $\underline \phi_0 \colon \underline K_0(A) \to \underline K_0(B)$ to be the isomorphism induced by $\phi_0$ and each $\phi_0^n$, and $\underline \phi_1 \colon \underline K_1(A) \to \underline K_1(B)$ to be the composition
\[
\underline K_1(A) \cong \bigoplus_{k\in \mathbb N} \underline K_1(A_k) \xrightarrow{\bigoplus_k \underline \phi_{1,k}} \bigoplus_{k\in \mathbb N} \underline K_1(B_k) \cong \underline K_1(B),
\]
where $\underline K_i(D) = K_i(D) \oplus \bigoplus_{n\in \mathbb N} K_i(D;\mathbb Z/n)$. It is straight forward to check that $\underline \phi = (\underline \phi_0 , \underline \phi_1) \colon \underline K(A) \to \underline K(B)$ is an isomorphism of $C(\widetilde{\mathbb N},\Lambda)$-modules.

If $A$ and $B$ are continuous $C(\widetilde{\mathbb N})$-algebras then $E(\widetilde{\mathbb N})$- and $KK(\widetilde{\mathbb N})$-theory agree by \cite[Theorem 5.4]{gabe:DM}.
\end{proof}

\end{document}